\documentclass[reqno]{amsart}
\usepackage{amssymb}
\usepackage{eucal}
\usepackage{amsmath}
\usepackage{amscd}
\usepackage[dvips]{color}
\usepackage{multicol}
\usepackage[all]{xy}           
\usepackage{graphicx}
\usepackage{color}
\usepackage{colordvi}
\usepackage{xspace}
\usepackage{bookmark}

\usepackage[active]{srcltx} 

\usepackage{hyperref}
\usepackage{lipsum}

\newcommand{\nc}{\newcommand}
\newcommand{\delete}[1]{}
\nc{\mfootnote}[1]{\footnote{#1}} 
\nc{\todo}[1]{\tred{To do:} #1}

\delete{
\nc{\mlabel}[1]{\label{#1}}  
\nc{\mcite}[1]{\cite{#1}}  
\nc{\mref}[1]{\ref{#1}}  
\nc{\mbibitem}[1]{\bibitem{#1}} 
}

\nc{\mlabel}[1]{\label{#1}  
{\hfill \hspace{1cm}{\bf{{\ }\hfill(#1)}}}}
\nc{\mcite}[1]{\cite{#1}{{\bf{{\ }(#1)}}}}  
\nc{\mref}[1]{\ref{#1}{{\bf{{\ }(#1)}}}}  
\nc{\mbibitem}[1]{\bibitem[\bf #1]{#1}} 

\newtheorem{theorem}{Theorem}[section]
\newtheorem{definition}[theorem]{Definition}
\newtheorem{lemma}[theorem]{Lemma}

\newtheorem{prop-def}[theorem]{Proposition-Definition}

\newtheorem{remark}[theorem]{Remark}

\nc{\tred}[1]{\textcolor{red}{#1}}
\nc{\tblue}[1]{\textcolor{blue}{#1}}
\nc{\tgreen}[1]{\textcolor{green}{#1}}
\nc{\tpurple}[1]{\textcolor{purple}{#1}}
\nc{\btred}[1]{\textcolor{red}{\bf #1}}
\nc{\btblue}[1]{\textcolor{blue}{\bf #1}}
\nc{\btgreen}[1]{\textcolor{green}{\bf #1}}
\nc{\btpurple}[1]{\textcolor{purple}{\bf #1}}

\nc{\li}[1]{\textcolor{red}{Xiaomin:#1}}
\nc{\cm}[1]{\textcolor{blue}{Chengming: #1}}

\nc{\twovec}[2]{\left(\begin{array}{c} #1 \\ #2\end{array} \right )}
\nc{\threevec}[3]{\left(\begin{array}{c} #1 \\ #2 \\ #3 \end{array}\right )}
\nc{\twomatrix}[4]{\left(\begin{array}{cc} #1 & #2\\ #3 & #4 \end{array} \right)}
\nc{\threematrix}[9]{{\left(\begin{matrix} #1 & #2 & #3\\ #4 & #5 & #6 \\ #7 & #8 & #9 \end{matrix} \right)}}
\nc{\twodet}[4]{\left|\begin{array}{cc} #1 & #2\\ #3 & #4 \end{array} \right|}

\nc{\rk}{\mathrm{r}}

\nc{\gensp}{V} 
\nc{\relsp}{\Lambda} 
\nc{\leafsp}{X}    
\nc{\treesp}{\overline{\calt}} 

\nc{\vin}{{\mathrm Vin}}    
\nc{\lin}{{\mathrm Lin}}    

\nc{\gop}{{\,\omega\,}}     
\nc{\gopb}{{\,\nu\,}}
\nc{\svec}[2]{{\tiny\left(\begin{matrix}#1\\
#2\end{matrix}\right)\,}}  
\nc{\ssvec}[2]{{\tiny\left(\begin{matrix}#1\\
#2\end{matrix}\right)\,}} 

\nc{\su}{\mathrm{Su}}
\nc{\tsu}{\mathrm{TSu}}
\nc{\TSu}{\mathrm{TSu}}
\nc{\eval}[1]{{#1}_{\big|D}}
\nc{\oto}{\leftrightarrow}

\nc{\oaset}{\mathbf{O}^{\rm alg}}
\nc{\omset}{\mathbf{O}^{\rm mod}}
\nc{\oamap}{\Phi^{\rm alg}}
\nc{\ommap}{\Phi^{\rm mod}}
\nc{\ioaset}{\mathbf{IO}^{\rm alg}}
\nc{\iomset}{\mathbf{IO}^{\rm mod}}
\nc{\ioamap}{\Psi^{\rm alg}}
\nc{\iommap}{\Psi^{\rm mod}}

\nc{\suc}{{successor}\xspace} \nc{\Suc}{{Successor}\xspace}
\nc{\sucs}{{successors}\xspace} \nc{\Sucs}{{Successors}\xspace}
\nc{\Tsuc}{{T-successor}\xspace}
\nc{\Tsucs}{{T-successors}\xspace} \nc{\Lsuc}{{L-successor}\xspace}
\nc{\Lsucs}{{L-successors}\xspace} \nc{\Rsuc}{{R-successor}\xspace}
\nc{\Rsucs}{{R-successors}\xspace}

\nc{\bia}{{$\mathcal{P}$-bimodule ${\bf k}$-algebra}\xspace}
\nc{\bias}{{$\mathcal{P}$-bimodule ${\bf k}$-algebras}\xspace}

\nc{\rmi}{{\mathrm{I}}}
\nc{\rmii}{{\mathrm{II}}}
\nc{\rmiii}{{\mathrm{III}}}

\nc{\pll}{\beta}
\nc{\plc}{\epsilon}

\nc{\ass}{{\mathit{Ass}}}
\nc{\lie}{{\mathit{Lie}}}
\nc{\comm}{{\mathit{Comm}}}
\nc{\dend}{{\mathit{Dend}}}
\nc{\zinb}{{\mathit{Zinb}}}
\nc{\tdend}{{\mathit{TDend}}}
\nc{\prelie}{{\mathit{preLie}}}
\nc{\postlie}{{\mathit{PostLie}}}
\nc{\quado}{{\mathit{Quad}}}
\nc{\octo}{{\mathit{Octo}}}
\nc{\ldend}{{\mathit{ldend}}}
\nc{\lquad}{{\mathit{LQuad}}}

 \nc{\adec}{\check{;}} \nc{\aop}{\alpha}
\nc{\dftimes}{\widetilde{\otimes}} \nc{\dfl}{\succ} \nc{\dfr}{\prec}
\nc{\dfc}{\circ} \nc{\dfb}{\bullet} \nc{\dft}{\star}
\nc{\dfcf}{{\mathbf k}} \nc{\apr}{\ast} \nc{\spr}{\cdot}
\nc{\twopr}{\circ} \nc{\tspr}{\star} \nc{\sempr}{\ast}
\nc{\disp}[1]{\displaystyle{#1}}
\nc{\bin}[2]{ (_{\stackrel{\scs{#1}}{\scs{#2}}})}  
\nc{\binc}[2]{ \left (\!\! \begin{array}{c} \scs{#1}\\
    \scs{#2} \end{array}\!\! \right )}  
\nc{\bincc}[2]{  \left ( {\scs{#1} \atop
    \vspace{-.5cm}\scs{#2}} \right )}  
\nc{\sarray}[2]{\begin{array}{c}#1 \vspace{.1cm}\\ \hline
    \vspace{-.35cm} \\ #2 \end{array}}
\nc{\bs}{\bar{S}} \nc{\dcup}{\stackrel{\bullet}{\cup}}
\nc{\dbigcup}{\stackrel{\bullet}{\bigcup}} \nc{\etree}{\big |}
\nc{\la}{\longrightarrow} \nc{\fe}{\'{e}} \nc{\rar}{\rightarrow}
\nc{\dar}{\downarrow} \nc{\dap}[1]{\downarrow
\rlap{$\scriptstyle{#1}$}} \nc{\uap}[1]{\uparrow
\rlap{$\scriptstyle{#1}$}} \nc{\defeq}{\stackrel{\rm def}{=}}
\nc{\dis}[1]{\displaystyle{#1}} \nc{\dotcup}{\,
\displaystyle{\bigcup^\bullet}\ } \nc{\sdotcup}{\tiny{
\displaystyle{\bigcup^\bullet}\ }} \nc{\hcm}{\ \hat{,}\ }
\nc{\hcirc}{\hat{\circ}} \nc{\hts}{\hat{\shpr}}
\nc{\lts}{\stackrel{\leftarrow}{\shpr}}
\nc{\rts}{\stackrel{\rightarrow}{\shpr}} \nc{\lleft}{[}
\nc{\lright}{]} \nc{\uni}[1]{\tilde{#1}} \nc{\wor}[1]{\check{#1}}
\nc{\free}[1]{\bar{#1}} \nc{\den}[1]{\check{#1}} \nc{\lrpa}{\wr}
\nc{\curlyl}{\left \{ \begin{array}{c} {} \\ {} \end{array}
    \right .  \!\!\!\!\!\!\!}
\nc{\curlyr}{ \!\!\!\!\!\!\!
    \left . \begin{array}{c} {} \\ {} \end{array}
    \right \} }
\nc{\leaf}{\ell}       
\nc{\longmid}{\left | \begin{array}{c} {} \\ {} \end{array}
    \right . \!\!\!\!\!\!\!}
\nc{\ot}{\otimes} \nc{\sot}{{\scriptstyle{\ot}}}
\nc{\otm}{\overline{\ot}}
\nc{\ora}[1]{\stackrel{#1}{\rar}}
\nc{\ola}[1]{\stackrel{#1}{\la}}
\nc{\pltree}{\calt^\pl}
\nc{\epltree}{\calt^{\pl,\NC}}
\nc{\rbpltree}{\calt^r}
\nc{\scs}[1]{\scriptstyle{#1}} \nc{\mrm}[1]{{\rm #1}}
\nc{\dirlim}{\displaystyle{\lim_{\longrightarrow}}\,}
\nc{\invlim}{\displaystyle{\lim_{\longleftarrow}}\,}
\nc{\mvp}{\vspace{0.5cm}} \nc{\svp}{\vspace{2cm}}
\nc{\vp}{\vspace{8cm}} \nc{\proofbegin}{\noindent{\bf Proof: }}
\nc{\proofend}{$\blacksquare$ \vspace{0.5cm}}
\nc{\freerbpl}{{F^{\mathrm RBPL}}}
\nc{\sha}{{\mbox{\cyr X}}}  
\nc{\ncsha}{{\mbox{\cyr X}^{\mathrm NC}}} \nc{\ncshao}{{\mbox{\cyr
X}^{\mathrm NC,\,0}}}
\nc{\shpr}{\diamond}    
\nc{\shprm}{\overline{\diamond}}    
\nc{\shpro}{\diamond^0}    
\nc{\shprr}{\diamond^r}     
\nc{\shpra}{\overline{\diamond}^r}
\nc{\shpru}{\check{\diamond}} \nc{\catpr}{\diamond_l}
\nc{\rcatpr}{\diamond_r} \nc{\lapr}{\diamond_a}
\nc{\sqcupm}{\ot}
\nc{\lepr}{\diamond_e} \nc{\vep}{\varepsilon} \nc{\labs}{\mid\!}
\nc{\rabs}{\!\mid} \nc{\hsha}{\widehat{\sha}}
\nc{\lsha}{\stackrel{\leftarrow}{\sha}}
\nc{\rsha}{\stackrel{\rightarrow}{\sha}} \nc{\lc}{\lfloor}
\nc{\rc}{\rfloor}
\nc{\tpr}{\sqcup}
\nc{\nctpr}{\vee}
\nc{\plpr}{\star}
\nc{\rbplpr}{\bar{\plpr}}
\nc{\sqmon}[1]{\langle #1\rangle}
\nc{\forest}{\calf}
\nc{\altx}{\Lambda_X} \nc{\vecT}{\vec{T}} \nc{\onetree}{\bullet}
\nc{\Ao}{\check{A}}
\nc{\seta}{\underline{\Ao}}
\nc{\deltaa}{\overline{\delta}}
\nc{\trho}{\tilde{\rho}}

\nc{\rpr}{\circ}
\nc{\dpr}{{\tiny\diamond}}
\nc{\rprpm}{{\rpr}}

\nc{\mmbox}[1]{\mbox{\ #1\ }} \nc{\ann}{\mrm{ann}}
\nc{\Aut}{\mrm{Aut}} \nc{\can}{\mrm{can}}
\nc{\twoalg}{{two-sided algebra}\xspace}
\nc{\colim}{\mrm{colim}}
\nc{\Cont}{\mrm{Cont}} \nc{\rchar}{\mrm{char}}
\nc{\cok}{\mrm{coker}} \nc{\dtf}{{R-{\rm tf}}} \nc{\dtor}{{R-{\rm
tor}}}

\nc{\depth}{{\mrm d}}
\nc{\Div}{{\mrm Div}} \nc{\End}{\mrm{End}} \nc{\Ext}{\mrm{Ext}}
\nc{\Fil}{\mrm{Fil}} \nc{\Frob}{\mrm{Frob}} \nc{\Gal}{\mrm{Gal}}
\nc{\GL}{\mrm{GL}} \nc{\Hom}{\mrm{Hom}} \nc{\hsr}{\mrm{H}}
\nc{\hpol}{\mrm{HP}} \nc{\id}{\mrm{id}} \nc{\im}{\mrm{im}}
\nc{\incl}{\mrm{incl}} \nc{\length}{\mrm{length}}
\nc{\LR}{\mrm{LR}} \nc{\mchar}{\rm char} \nc{\NC}{\mrm{NC}}
\nc{\mpart}{\mrm{part}} \nc{\pl}{\mrm{PL}}
\nc{\ql}{{\QQ_\ell}} \nc{\qp}{{\QQ_p}}
\nc{\rank}{\mrm{rank}} \nc{\rba}{\rm{RBA }} \nc{\rbas}{\rm{RBAs }}
\nc{\rbpl}{\mrm{RBPL}}
\nc{\rbw}{\rm{RBW }} \nc{\rbws}{\rm{RBWs }} \nc{\rcot}{\mrm{cot}}
\nc{\rest}{\rm{controlled}\xspace}
\nc{\rdef}{\mrm{def}} \nc{\rdiv}{{\rm div}} \nc{\rtf}{{\rm tf}}
\nc{\rtor}{{\rm tor}} \nc{\res}{\mrm{res}} \nc{\SL}{\mrm{SL}}
\nc{\Spec}{\mrm{Spec}} \nc{\tor}{\mrm{tor}} \nc{\Tr}{\mrm{Tr}}
\nc{\mtr}{\mrm{sk}}

\nc{\ab}{\mathbf{Ab}} \nc{\Alg}{\mathbf{Alg}}
\nc{\Algo}{\mathbf{Alg}^0} \nc{\Bax}{\mathbf{Bax}}
\nc{\Baxo}{\mathbf{Bax}^0} \nc{\RB}{\mathbf{RB}}
\nc{\RBo}{\mathbf{RB}^0} \nc{\BRB}{\mathbf{RB}}
\nc{\Dend}{\mathbf{DD}} \nc{\bfk}{{\bf k}} \nc{\bfone}{{\bf 1}}
\nc{\base}[1]{{a_{#1}}} \nc{\detail}{\marginpar{\bf More detail}
    \noindent{\bf Need more detail!}
    \svp}
\nc{\Diff}{\mathbf{Diff}} \nc{\gap}{\marginpar{\bf
Incomplete}\noindent{\bf Incomplete!!}
    \svp}
\nc{\FMod}{\mathbf{FMod}} \nc{\mset}{\mathbf{MSet}}
\nc{\rb}{\mathrm{RB}} \nc{\Int}{\mathbf{Int}}
\nc{\Mon}{\mathbf{Mon}}
\nc{\remarks}{\noindent{\bf Remarks: }}
\nc{\OS}{\mathbf{OS}} 
\nc{\Rep}{\mathbf{Rep}}
\nc{\Rings}{\mathbf{Rings}} \nc{\Sets}{\mathbf{Sets}}
\nc{\DT}{\mathbf{DT}}

\nc{\BA}{{\mathbb A}} \nc{\CC}{{\mathbb C}} \nc{\DD}{{\mathbb D}}
\nc{\EE}{{\mathbb E}} \nc{\FF}{{\mathbb F}} \nc{\GG}{{\mathbb G}}
\nc{\HH}{{\mathbb H}} \nc{\LL}{{\mathbb L}} \nc{\NN}{{\mathbb N}}
\nc{\QQ}{{\mathbb Q}} \nc{\RR}{{\mathbb R}} \nc{\BS}{{\mathbb{S}}} \nc{\TT}{{\mathbb T}}
\nc{\VV}{{\mathbb V}} \nc{\ZZ}{{\mathbb Z}}

\nc{\calao}{{\mathcal A}} \nc{\cala}{{\mathcal A}}
\nc{\calc}{{\mathcal C}} \nc{\cald}{{\mathcal D}}
\nc{\cale}{{\mathcal E}} \nc{\calf}{{\mathcal F}}
\nc{\calfr}{{{\mathcal F}^{\,r}}} \nc{\calfo}{{\mathcal F}^0}
\nc{\calfro}{{\mathcal F}^{\,r,0}} \nc{\oF}{\overline{F}}
\nc{\calg}{{\mathcal G}} \nc{\calh}{{\mathcal H}}
\nc{\cali}{{\mathcal I}} \nc{\calj}{{\mathcal J}}
\nc{\call}{{\mathcal L}} \nc{\calm}{{\mathcal M}}
\nc{\caln}{{\mathcal N}} \nc{\calo}{{\mathcal O}}
\nc{\calp}{{\mathcal P}} \nc{\calq}{{\mathcal Q}} \nc{\calr}{{\mathcal R}}
\nc{\calt}{{\mathcal T}} \nc{\caltr}{{\mathcal T}^{\,r}}
\nc{\calu}{{\mathcal U}} \nc{\calv}{{\mathcal V}}
\nc{\calw}{{\mathcal W}} \nc{\calx}{{\mathcal X}}
\nc{\CA}{\mathcal{A}}

\nc{\fraka}{{\mathfrak a}} \nc{\frakB}{{\mathfrak B}}
\nc{\frakb}{{\mathfrak b}} \nc{\frakd}{{\mathfrak d}}
\nc{\oD}{\overline{D}}
\nc{\frakF}{{\mathfrak F}} \nc{\frakg}{{\mathfrak g}}
\nc{\frakm}{{\mathfrak m}} \nc{\frakM}{{\mathfrak M}}
\nc{\frakMo}{{\mathfrak M}^0} \nc{\frakp}{{\mathfrak p}}
\nc{\frakS}{{\mathfrak S}} \nc{\frakSo}{{\mathfrak S}^0}
\nc{\fraks}{{\mathfrak s}} \nc{\os}{\overline{\fraks}}
\nc{\frakT}{{\mathfrak T}}
\nc{\oT}{\overline{T}}
\nc{\frakX}{{\mathfrak X}} \nc{\frakXo}{{\mathfrak X}^0}
\nc{\frakx}{{\mathbf x}}
\nc{\frakTx}{\frakT}      
\nc{\frakTa}{\frakT^a}        
\nc{\frakTxo}{\frakTx^0}   
\nc{\caltao}{\calt^{a,0}}   
\nc{\ox}{\overline{\frakx}} \nc{\fraky}{{\mathfrak y}}
\nc{\frakz}{{\mathfrak z}} \nc{\oX}{\overline{X}}

\font\cyr=wncyr10

\nc{\redtext}[1]{\textcolor{red}{#1}}

\makeatletter
\g@addto@macro{\endabstract}{\@setabstract}
\newcommand{\authorfootnotes}{\renewcommand\thefootnote{\@fnsymbol\c@footnote}}%
\makeatother
\begin{document}
\begin{center}
  \LARGE
\textbf{Biderivations of W-algebra $W(2,2)$ and Virasoro algebra without skewsymmetric condition and their applications}

  \normalsize
  \authorfootnotes
Xiaomin Tang   \footnote{Corresponding author: {\it X. Tang. Email:} x.m.tang@163.com}
\par \bigskip

   \textsuperscript{1}Department of Mathematics, Heilongjiang University,
Harbin, 150080, P. R. China   \par

\end{center}


\begin{abstract}

In this paper, we characterize the biderivations of W-algebra $W(2,2)$ and Virasoro algebra $Vir$ without skewsymmetric condition. We get two classes of non-inner biderivations. As applications, we characterize the forms of linear commuting maps  and the commutative post-Lie algebra structures on W-algebra $W(2,2)$ and Virasoro algebra $Vir$.

\vspace{2mm}

\noindent{\it Keywords:} biderivation, Virasoro algebra, W-algebra, linear commuting map, post-Lie algebra

\noindent{\it AMS subject classifications:} 17B05, 17B40, 17B65.

\end{abstract}

\setcounter{section}{0}
{\ }

 \baselineskip=20pt

\section{Introduction and  preliminary results}

Derivations and generalized derivations are very important subjects in the
study of both algebras and their generalizations. In recent years, many authors put so much effort into this problems \cite{Bre1995,Ben2009,Chen2016,Du2013,Gho2013,Hanw,tang2016,WD1,WD3,WD2}. In \cite{Bre1995}, Bre$\breve{s}$ar et. al. showed that all biderivations on commutative prime rings are inner biderivations, and determine the biderivations of semiprime
rings. The notation of biderivations of Lie algebras was introduced in \cite{WD3}, his part result is shown that the skewsymmetric biderivation of finite dimensional complex simple Lie algebra is inner. In addition, \cite{WD1} and \cite{Chen2016} study the skewsymmetric biderivations on the Schr$\ddot{o}$dinger-Virasoro algebra and simple generalized Witt algebra over a field of characteristic $0$, respectively. It is shown that thy are inner.  In \cite{Hanw}, the authors determine all the skewsymmetric biderivations of $W(a,b)$ and find that there exist non-inner biderivations. The problem on biderivation is also generated to super-biderivation, in \cite{WD2} the authors show that the skewsymmetric super-biderivations on the Super-Virasoro Algebras is inner.

Note that almost all the above articles on Lie (super-)algebra are assumed that the biderivation is (super-)skewsymmetric (although the authors in \cite{WD1} and \cite{WD3} did not refer to the skewsymmetric hypothesis, but this hypothesis should be put into use, as pointed out in \cite{Chen2016}, \cite{tang2016}). So, this class of problems without (super-)skewsymmetric condition should receive attention. Based on this, we \cite{tang2016} study the biderivations of finite dimensional complex simple Lie algebra and general linear Lie algebra without the restriction of skewsymmetric, we find that there exist non-skewsymmetric biderivations. It may be useful and interesting for computing the biderivations of some important Lie (super-)algebras. On the other hand, Virasoro algebra as the universal central extension of the Witt algebra, is an infinite dimensional Lie algebra which has appeared in several context, it has been paid great attention of mathematicians and physicists, see \cite{Frec,GKO1,GKO2} and so on. Another
an infinite dimensional Lie algebra $W(2,2)$ is a class of W-algebra, which is a non-central extension of the Virasoro algebra and first introduced by \cite{JiangDong} in their recent work on the classification of some simple vertex operator algebras, and then some scholars studied the theory on structures and representations of $W(2,2)$, see \cite{Chenhj,GJP,Jiangzhangwei,Rad} and so forth. In this paper, we dedicate themselves to study the biderivations of Virasoro algebra and W-algebra $W(2,2)$ without skewsymmetric condition.   And then, we also give two applications of biderivations of them. That is, the fist application is to characterize the linear commuting map on a Lie algebra by using a well known method, which also works for skewsymmetric biderivation.  The second application is to characterize the commutative post-Lie algebra structures on Lie algebra, note that the precondition of the method is that the biderivation should be assumed to be non-skewsymmetric.

 Although paper we work under complex number field $\mathbb{C}$, but it  also works with algebraically closed field of characteristic zero.
For an arbitrary Lie algebra $L$, we recall a bilinear map $f : L\times L \rightarrow L$ is a
biderivation of $L$ if it is a derivation with respect to both components, or to be more precise, one has
\begin{definition}
Suppose that $L$ is a Lie algebra. A bilinear map $f: L\times L\rightarrow L$ is called a biderivation if it satisfying
\begin{eqnarray}
f([x,y],z)=[x,f(y,z)]+[f(x,z),y], \label{2der}\\
f(x,[y,z])=[f(x,y),z]+[y,f(x,z)] \label{1der}
\end{eqnarray}
for all $x, y, z\in L$.
\end{definition}
For a complex number $\lambda$, we define a bilinear map $f: L\times L\rightarrow L$ given by $f(x,y)=\lambda [x,y]$, then it is easy to verify that $f$
is a biderivation of $L$. We call such biderivation to be inner.  $f$ is called skew-symmetric if $f(x,y)=-f(y,x)$ for all $x,y\in L$.

Now let us recall the definition of derivation of Lie algebra as follows.
\begin{definition}
Suppose that $L$ is a Lie algebra. A linear map $\phi: L\rightarrow L$ is called a derivation if it satisfying
\begin{eqnarray}
\phi([x,y])=[\phi(x),y]+[x,\phi(y)]
\end{eqnarray}
for all $x, y\in L$.
\end{definition}

For $x\in L$, it is easy to see that $\phi_x:L\rightarrow L, y\mapsto {\rm ad} x(y)=[x,y], $ for all $y\in L$ is a derivation of $L$, which is called an inner derivation. Below let us recall the definition of Virasoro algebra and W-algebra $W(2,2)$.

\begin{definition}
The infinite dimensional Lie algebra $Vir$ is called Virasoro algebra if it with the basis  $\{L_m,  c| m\in \mathbb{Z} \}$ and Lie brackets
\begin{equation*}
[L_m,L_n]=(m-n)L_{m+n}+\frac{m^3-m}{12}\delta_{m+n,0}c, \ \ [Vir, c]=0.
\end{equation*}
\end{definition}

\begin{definition}
The W-algebra $W(2,2)$ is a Lie algebra with the basis $\{L_m, H_m, c| m\in \mathbb{Z} \}$ and Lie brackets
\begin{eqnarray*}
&[L_m,L_n]=(m-n)L_{m+n}+\frac{m^3-m}{12}\delta_{m+n,0}c,\\
&[L_m,H_n]=(m-n)H_{m+n}+\frac{m^3-m}{12}\delta_{m+n,0}c,\\
&[H_m,H_n]=0, \ \ [W(2,2), c]=0.
\end{eqnarray*}
\end{definition}

We have the following lemmas about the derivation of Virasoro algebra and $W(2,2)$.
\begin{lemma} \label{zhulinsheng}\cite{zhulin}
Every derivation of Virasoro algebra is inner.
\end{lemma}\label{dervira}
\begin{lemma}\cite{GJP}\label{innerLie}
Every derivation $\delta$ of $W(2,2)$ is of the following form
$$
\delta = {\rm ad} x + a D
$$
for some $x\in W(2,2)$ and $a\in \mathbb{C}$, where $D$ is a out derivation of $W(2,2)$, which is defined by $D(L_m)=0$, $D(I_m)=I_m$ for all $m\in \mathbb{Z}$ and $D(c)=dc$ for some $d\in \mathbb{C}$.
\end{lemma}

\section{Biderivation  of W-algebra}

In this section, we always assume that $f$ is a biderivation of W-algebra $W(2,2)$.
\begin{lemma}\label{phipsi}
There are two linear maps $\phi$ and $\psi$ from $W(2,2)$ into itself such that
\begin{equation}
f(x,y)=l_x D(y)+[\phi(x), y]=r_y D(x)+[x, \psi(y)], \  \ \forall x,y\in L,
\end{equation}
where $l_x, r_x$ are complex numbers depend on $x$, and $D$ is given by Lemma \ref{innerLie}.
\end{lemma}

\begin{proof}
For the biderivation $f$ of $L$ and a fixed element $x\in L$, we define a map $\phi_x: W(2,2)\rightarrow W(2,2)$ is given by $\phi_x(y)=f(x,y)$. Then we know by (\ref{1der}) that $\phi_x$ is a derivation of $W(2,2)$. So there is a map
$\phi: L\rightarrow L$ such that $\phi_x=k_x D+{\rm ad}\phi(x)$, i.e.,  $f(x,y)=l_x D(y)+[\phi(x), y]$. Due to $f$ is bilinear, one has $\phi$ is linear. Similarly, as we define a map $\psi_z$ from $W(2,2)$ into itself is given by $\psi_z(y)=f(y, z)$ for all $y\in L$, we can get a linear map $\psi$ from $W(2,2)$ into itself such that $f(x,y)=r_y D(x)+{\rm ad}(-\psi (y))(x)=r_y D(x)+[x, \psi(y)]$. The proof is completed.
\end{proof}

By Lemmas \ref{phipsi} and \ref{innerLie}, we have
\begin{lemma}\label{relation}
For any $i,j\in \mathbb{Z}$, one has the following equations:
\begin{eqnarray}
f(L_i,L_j)=[\phi(L_i),L_j]=[L_i,\psi(L_j)], \label{LL1}\\
f(L_i,H_j)=l_{L_i} H_j+[\phi(L_i),H_j]=[L_i,\psi(H_j)],\label{LL2}\\
f(H_i,L_j)=[\phi(H_i),L_j]=r_{L_j}H_i+[H_i,\psi(L_j)],\label{LL3}\\
f(H_i,H_j)=l_{H_i}H_j+[\phi(H_i), H_j]=r_{H_j}H_i+[H_i,\psi(H_j)], \label{LL4}\\
f(x,c)=dl_{x}c, f(c,y)=dr_yc, \ x,y\in \{L_m,H_m\},\label{LL5}
\end{eqnarray}
where $\phi,\psi, l_x, r_x$ are given by Lemma \ref{phipsi}.
\end{lemma}

\begin{lemma}\label{LMLN}
Let $\phi$ and $\psi$ be defined by Lemma \ref{phipsi}.  For any $m\in \mathbb{Z}$, there are $\lambda, \mu, \alpha_m \beta_m\in C$ such that
\begin{eqnarray*}
\phi(L_m)=\lambda L_m + \mu H_m+ \alpha_m c, \\
\psi(L_m)=\lambda L_m + \mu H_m+ \beta_m c.
\end{eqnarray*}
\end{lemma}

\begin{proof}
For any but fixed $n\in \mathbb{Z}$, let
\begin{eqnarray}
\phi(L_n)=\sum_{i\in \mathbb{Z}} k_{i}^{(n)}L_i+\sum_{i\in \mathbb{Z}} t_{i}^{(n)}H_i+\alpha_n c, \label{ee1}\\
\psi(L_{n})=\sum_{i\in \mathbb{Z}} h_{i}^{(n)}L_i+\sum_{i\in \mathbb{Z}} g_{i}^{(n)}H_i+\beta_n c, \label{ee3}
\end{eqnarray}
where $k_{i}^{(n)}, t_{i}^{(n)},  h_{i}^{(n)}, g_{i}^{(n)}, \alpha_n, \beta_n \in \mathbb{C}, i\in \mathbb{Z}$.
Therefore, we have
\begin{equation}\label{philmln}
[\phi (L_m), L_n]=\sum_{i\in \mathbb{Z}} (i-n)k_{i}^{(m)}L_{n+i}+\sum_{i\in \mathbb{Z}}(i-n)t_{i}^{(m)}H_{n+i}+\frac {n-n^3}{12}(k_{-n}^{(m)}+t_{-n}^{(m)})c,
\end{equation}
\begin{eqnarray}
[L_m,\psi(L_n)]=\sum_{i\in \mathbb{Z}} (m-i)h_{i}^{(n)}L_{m+i}+
\sum_{i\in \mathbb{Z}}(m-i)g_{i}^{(n)}H_{m+i}+\frac {m^3-m}{12}(h_{-m}^{(n)}+g_{-m}^{(n)})c \nonumber\\
=\sum_{i\in \mathbb{Z}} (2m-n-i)h_{n-m+i}^{(n)}L_{n+i}+
\sum_{i\in \mathbb{Z}}(2m-n-i)g_{n-m+i}^{(n)}H_{n+i}+\frac {m^3-m}{12}(h_{-m}^{(n)}+g_{-m}^{(n)})c. \label{lmpsiln}
\end{eqnarray}
By (\ref{LL1}), we have $f(L_m,L_n)=[\phi (L_m), L_n]=[L_m,\psi(L_n)]$.  This, together whith Equations (\ref{philmln}) and (\ref{lmpsiln}), implies that
\begin{eqnarray}
(i-n)k_i^{(m)}=(2m-n-i)h_{n-m+i}^{(n)}, \label{abcd1}\\
(i-n)t_i^{(m)}=(2m-n-i)g_{n-m+i}^{(n)}. \label{abcd2}
\end{eqnarray}
 for all $m,n,i\in \mathbb{Z}$.  For any $m,n$ with $m\neq n$, by taking $i= 2m-n, n, m $  in (\ref{abcd1}) respectively, we have
 \begin{equation}\label{kkk}
 k_{2m-n}^{(m)}=h_{2n-m}^{(n)}=0, \ k_{m}^{(m)}=h_{n}^{(n)}, \ \forall m, n\in \mathbb{Z} \ \text {with}\ m\neq n.
 \end{equation}
 Let $m, n$ run all integers with $m\neq n$, we conclude by (\ref{kkk}) that
 \begin{equation}\label{kkklll}
 k_{i}^{(m)}=0, h_{j}^{(n)}=0,  \forall i\neq m, j\neq n,
 \end{equation}
 and
\begin{equation}\label{kkkrrr}
k_{n}^{(n)}=h_{n}^{(n)}=\ h_{0}^{(0)},\ \forall n\in \mathbb{Z}.
\end{equation}

 Similarly, we have by (\ref{abcd2}) that
$ t_{2m-n}^{(m)}=g_{2n-m}^{(n)}=0, \ t_{m}^{(m)}=g_{n}^{(n)}$ for all $m, n\in \mathbb{Z}$ with $m\neq n$. And then, we deduce
\begin{equation}\label{kkkll}
 t_{i}^{(m)}=0, g_{j}^{(n)}=0,  \forall i\neq m, j\neq n,
 \end{equation}
 and
\begin{equation}\label{kkkrr}
t_{n}^{(n)}=g_{n}^{(n)}=\ t_{0}^{(0)},\ \forall n\in \mathbb{Z}.
\end{equation}
 By letting $\ h_{0}^{(0)}=\lambda$ and $t_{0}^{(0)}=\mu$, the proof follows from Equations (\ref{ee1}), (\ref{ee3}), (\ref{kkklll}), (\ref{kkkrrr}),(\ref{kkkll}) and (\ref{kkkrr}).
\end{proof}

\begin{lemma}\label{KT}
Let $\phi$ and $\psi$ be defined by Lemma \ref{phipsi}. For any $m\in \mathbb{Z}$, we have
 \begin{eqnarray*}
 \phi(H_m)=\lambda H_m + \gamma_m c,\\
 \psi(H_m)=\lambda H_m+ \eta_m c, \\
 l_{L_m}=r_{L_m}=l_{H_m}=r_{H_m}=0,
 \end{eqnarray*}
 where $\lambda$ is given by Lemma \ref{LMLN}, $ l_{L_m}, r_{L_m}, l_{H_m}, r_{H_m}$ are give by Lemma \ref{LMLN} and $\gamma_m,\eta_m\in \mathbb{C}$.
\end{lemma}
\begin{proof}
For any $n\in \mathbb{Z}$, let
\begin{eqnarray}
\phi(H_n)=\sum_{i\in \mathbb{Z}} p_{i}^{(n)}L_i+\sum_{i\in \mathbb{Z}} q_{i}^{(n)}H_i+\gamma_n c, \label{ee2}\\
\psi(H_{n})=\sum_{i\in \mathbb{Z}} s_{i}^{(n)}L_i+\sum_{i\in \mathbb{Z}} r_{i}^{(n)}H_i+ \eta_n c, \label{ee4}
\end{eqnarray}
where $p_{i}^{(n)}, q_{i}^{(n)},  s_{i}^{(n)}, r_{i}^{(n)}, \gamma_n, \eta_n \in \mathbb{C}, i\in \mathbb{Z}$.
By the direct computations, we deduce by Lemma \ref{LMLN} that
\begin{equation}\label{philmhn}
[\phi (L_m), H_n]=(m-n)\lambda H_{m+n}+ \frac {m^3-m}{12} \lambda\delta_{m+n,0} c,
\end{equation}
\begin{equation}\label{hmpsiln}
[H_m, \psi (L_n)]=(m-n)\lambda H_{m+n}- \frac {n^3-n}{12} \lambda \delta_{m+n,0} c
\end{equation}
and by Equations (\ref{ee4}), (\ref{ee3}) that
\begin{equation}\label{lmpsihn}
[L_m, \psi (H_n)]=\sum_{i\in \mathbb{Z}} (m-i) s_{i}^{(n)}L_{m+i}+\sum_{i\in \mathbb{Z}} (m-i)r_{i}^{(n)}H_{m+i}+
\frac {m^3-m}{12}(s_{-m}^{(n)}+r_{-m}^{(n)})c,
\end{equation}
\begin{equation}\label{phihmln}
[\phi(H_m), L_n]=\sum_{i\in \mathbb{Z}}(i-n) p_{i}^{(m)}L_{n+i}+\sum_{i\in \mathbb{Z}}(i-n) q_{i}^{(m)}H_{n+i}+
\frac {n-n^3}{12}(p_{-n}^{(m)}+q_{-n}^{(m)})c.
\end{equation}
Now applying (\ref{LL2}) one has
\begin{equation}\label{flmhn}
f(L_m,H_n)=l_{L_m} H_n+[\phi(L_m),H_n]=[L_m,\psi(H_n)].
\end{equation}
 This, together with (\ref{philmhn}), (\ref{lmpsihn}), yields
$(m-i) s_{i}^{(n)}=0$ for all $m,n,i\in \mathbb{Z}$. It follows that
\begin{equation}\label{si=0}
s_{i}^{(n)}=0, \ \forall i \in \mathbb{Z}.
\end{equation}
Equations (\ref{philmhn}), (\ref{lmpsihn}) and (\ref{flmhn}) also tell us that $(m-n)r_{n}^{(n)}=(m-n)\lambda$ while $m\neq n$, which gives that
$r_{n}^{(n)}=\lambda$.  Now we again review Equations (\ref{philmhn}), (\ref{lmpsihn}), (\ref{flmhn}) and (\ref{si=0}), one has
$$
l_{L_m} H_n+\frac {m^3-m}{12} \lambda\delta_{m+n,0} c=\sum_{i\in \mathbb{Z}\setminus\{n\}} (m-i)r_{i}^{(n)}H_{m+i}+
\frac {m^3-m}{12}(s_{-m}^{(n)}+r_{-m}^{(n)})c.
$$
Compare the two hands of the above equation, we have $(m-i)r_{i}^{(n)}=0$ for any $i\in \mathbb{Z}$ with $i\neq m-n$ and $i\neq n$.  Above all, we get
$$
r_{n}^{(n)}=\lambda, \ \ r_{i}^{(n)}=0, \forall i\neq n.
$$
We also have $l_{L_m}=0$ for all $m\in \mathbb{Z}$. This, together with (\ref{si=0}) and (\ref{ee4}), implies $\psi(H_m)=\lambda H_m+ \eta_m c$.
Similarly, we have by Equations (\ref{LL3}), (\ref{ee2}), (\ref{hmpsiln}) and (\ref{phihmln}) that $r_{L_m}=0$ and
$\phi(H_m)=\lambda H_m+ \gamma_m c$ for all $m\in \mathbb{Z}$.   From this, applying (\ref{LL4}), we deduce
\begin{eqnarray*}
f(H_m, H_n)&=&l_{H_m}H_n+[\phi(H_m), H_n]=l_{H_m}H_n\\
&=&r_{H_n}H_m+[H_m, \psi(H_n)]=r_{H_n}H_m.
\end{eqnarray*}
The above equation implies $l_{H_m}=r_{H_n}=0$ for all $m,n\in \mathbb{Z}$ with $m\neq n$. This completes the proof.
\end{proof}

\begin{lemma} \label{c=0}
For any $m\in \mathbb{Z}$, we have
$f(c,x)=f(x,c)=0, \forall x\in \{L_m,I_m\}$ and $f(c,c)=0$.
\end{lemma}

\begin{proof}
Lemma \ref{KT} tells us that $l_{L_m}=r_{L_m}=l_{H_m}=r_{H_m}=0$ for all $m\in \mathbb{Z}$. This, together with (\ref{LL5}), gives $f(c,x)=f(x,c)=0, \forall x\in \{L_m,I_m\}$. On the other hand, note that $c=2[L_2,L_{-2}]-8L_0$, we have
$$
f(c,c)=f(2[L_2,L_{-2}]-8L_0, c)=2[f(L_2,c), L_{-2}]+2[L_2, f(L_{-2}, c)]-8f(L_0,c)=0.
$$
The proof is completed.
\end{proof}

For convenience, we define a linear map $\omega: W(2,2)\rightarrow W(2,2)$ is given by
$\omega(L_m)=H_m, \omega (H_m)=\omega (c)=0$ for all $m\in \mathbb{Z}$, that is,
$$
\omega (\sum_{i\in \mathbb{Z}} k_i L_i+\sum_{i\in \mathbb{Z}} t_i H_i+ r c)=\sum_{i\in \mathbb{Z}} k_i H_i.
$$
For all $m,n\in \mathbb{Z}$, it is obvious that $[\omega(L_m),H_n]=[L_m, \omega(H_n)]=0$, $[\omega(H_m),L_n]=[H_m, \omega(L_n)]=0$. Note that $[\omega(L_m),L_n]=[H_m,L_n]=(m-n)H_{m+n}+\frac{n-n^3}{12}\delta_{m+n,0}c$ and $[L_m,\omega(L_n)]=[L_m,H_n]=(m-n)H_{m+n}+\frac{m^3-m}{12}\delta_{m+n,0}c$. We deduce $[\omega(L_m),L_n]-[L_m,\omega(L_n)]=\frac{n+m-(n^3+m^3)}{12}\delta_{m+n,0}c=0$, and then $[\omega(L_m),L_n]=[L_m,\omega(L_n)]$. The above conclusions along with $[\omega(x), c]=[x, \omega(c)]=[\omega (c), y]=[c, \omega (y)]=0$ for all $x,y\in W(2,2)$, yield that
\begin{equation}\label{omega}
[\omega(x), y]=[x, \omega (y)], \ \forall x,y\in W(2,2).
\end{equation}

We now state our main result as follows.

\begin{theorem}\label{maintheo}
$f$ is a biderivation of $W(2,2)$ if and only if there are two complex number $\lambda, \mu$ such that
\begin{equation}\label{fform}
f(x,y)=\lambda [x,y]+\mu[\omega (x),y], \ \ \forall x,y\in W(2,2),
\end{equation}
where $\omega$ is a linear map from $W(2,2)$ into it self defined by $\omega(L_m)=H_m, \omega (H_m)=\omega (c)=0$ for all $m\in \mathbb{Z}$.
\end{theorem}
\begin{proof}
 Suppose that $f$ has the form (\ref{fform}). By Jacobin identity and (\ref{omega}), we have
$$[\omega([x,y]),z]=[[x,y],\omega(z)]=[x,[y,\omega(z)]]+[[x,\omega(z)],y]=[x,[\omega (y),z]]+[[\omega (x),z],y],$$
$$[\omega(x),[y,z]]=[[\omega(x),y],z]+[y,[\omega(x),z]],$$
for all $x,y\in W(2,2)$. From this, it is easy to verify that $f$ is a biderivation.

Conversely, suppose that $f$ is a biderivation of $W(2,2)$.  It follows by Lemmas \ref{relation}, \ref{KT} and \ref{c=0} that
\begin{equation}\label{woaini}
f(x,y)=[\phi(x),y], \ \forall x,y\in W(2,2),
\end{equation}
where $\phi$ is given by Lemma \ref{phipsi} after promising $\phi(c)=\lambda c$. Now by the conclusions of Lemmas \ref{LMLN}, \ref{KT} and \ref{c=0}, we have
\begin{eqnarray*}
\phi(L_m)=(\lambda \varepsilon_{W(2,2)}+\mu \omega) (L_m)+\alpha_m c,\\
\phi(I_m)=(\lambda \varepsilon_{W(2,2)}+\mu \omega) (H_m)+\gamma_m c,\\
\phi(c)=(\lambda \varepsilon_{W(2,2)}+\mu \omega) (c),
\end{eqnarray*}
where $\varepsilon_{W(2,2)}$ denote the identity map of $W(2,2)$. Hence, we can assume that
$$
\phi(x)=(\lambda \varepsilon_{W(2,2)}+\mu \omega) (x)+\sigma_x c, \ \forall x\in W(2,2),
$$
for some $\sigma_x\in \mathbb{C}$. So it follows from (\ref{woaini}) that $f(x,y)=[(\lambda \varepsilon_{W(2,2)}+\mu \omega) (x)+\sigma_x c,y]=\lambda [x,y]+\mu [\omega(x),y]$.
The proof is completed.
\end{proof}

\section{Biderivation of Virasoro algebra}

In this section, let $f$ be a biderivation of $Vir$. We give the following theorem.

\begin{theorem}\label{theo-virasoro}
$f$ is a derivation of $Vir$ if and only if $f$ is inner, i.e., there exist a complex number $\lambda$ such that
\begin{equation}\label{ffform}
f(x,y)=\lambda [x,y], \ \ \forall x,y\in Vir,
\end{equation}
\end{theorem}

\begin{proof}
The ``if part'' is easy to verify. We now prove the ``only if'' part.

Firstly, by Lemma \ref{zhulinsheng} we know that every derivation of $Vir$ is inner. If we let $\phi_x(y)=f(x,y)=\psi_y(x)$ for all $x,y\in gl_n(\mathbb{C})$, then $\phi_x$ and $\psi_y$ are both derivations of $Vir$. Similar to the proof of Lemma \ref{phipsi}, there are linear maps $\phi, \psi$ from $Vir$ into itself such that
$$
f(x,y)=[\phi(x),y]=[x,\psi(y)], \ \forall x,y\in Vir.
$$
We mark $H_i=0$ for all $i\in \mathbb{Z}$ in the proof of Lemma \ref{LMLN}, and using the same way of the proof of Lemma \ref{LMLN}, it is not difficult to see that
\begin{equation}\label{abc}
\phi(L_m)=\lambda L_m + \alpha_m c, \ \
\psi(L_m)=\lambda L_m + \beta_m c, \ \forall m\in \mathbb{Z},
\end{equation}
for some $\lambda, \alpha_m, \beta_m\in \mathbb{C}$.

Nextly, we claim that $\rho(c)\in \mathbb{C}c$ for any derivation $\rho$ of $Vir$. In fact, note that $[x,c]=0$, one has $0=\rho([x,c])=[\rho(x),c]+[x,\rho(c)]=[x,\rho(c)]$ for all $x\in Vir$, the claim follows since $Z(Vir)=\mathbb{C}c$.  Thus, note that $\phi_x=f(x, \cdot)$ and $\psi_y=f(\cdot, y)$ are both derivations of $Vir$, applying the above claim we have
\begin{equation}\label{incent}
f(x,c), f(c,y)\in \mathbb{C}c, \ \ \forall x,y\in Vir.
\end{equation}
Note that $[L_2,L_{-2}]=4L_0+\frac{2^3-2}{12}c$ and $[L_1,L_{-1}]=2L_0$, it follows that $c=2[L_2,L_{-2}]-4[L_1,L_{-1}]$. We also have $L_m =\frac{1}{m}[L_m,L_0]$ for all $m\in \mathbb{Z}\setminus \{0\}.$
Thus, we have by Equations (\ref{incent}) and (\ref{2der}) that
$$
f(L_m, c)=f(\frac{1}{m}[L_m,L_0], c)=\frac{1}{m}([L_m, f(L_0, c)]+[f(L_m, c), L_0])=0,\ \ m\neq 0,
$$
$$
f(L_0,c)=f(\frac{1}{2}[L_1,L_{-1}], c)=\frac{1}{2}([L_1, f(L_{-1}, c)]+[f(L_1, c), L_{-1}])=0,
$$
\begin{eqnarray*}
&&f(c,c)=f(2[L_2,L_{-2}]-4[L_1,L_{-1}],c)\\
&=&2([L_2, f(L_{-2},c)]+[f(L_2,c),L_{-2}])-4([L_1, f(L_{-1},c)]+[f(L_1,c),L_{-1}])=0.
\end{eqnarray*}
Similarly, we have by Equations (\ref{incent}) and (\ref{1der}) that $f(c, L_m)= f(c, L_0)=0$ for all $m\neq 0$.
The above discussion tells us that $f(x,c)=f(c,x)=0$ for all $x\in Vir$. This allows us to assume that $\phi(c)=\psi(c)=\lambda c$ and so that
$f(c,x)=[\phi(c),x]$ and $f(x,c)=[x,\psi(c)]$ just establish.  Again according to (\ref{abc}), we are able to assume that
$\phi(x)=\lambda x+k_x c, \ \ \forall x\in Vir$, where $k_x\in \mathbb{C}$.

Finally,  we conclude that
$$
f(x,y)=[\phi(x),y]=[\lambda x+k_x c, y]=\lambda [x, y],
$$
for all $x,y\in Vir$. The proof is completed.
\end{proof}

\section{Applications}

In this section, we give two applications of biderivation.

\subsection{Linear commuting maps on Lie algebras}

Recall that a linear commuting map $\phi$ on a Lie algebra $L$ subject to $[\phi(x),x]=0$ for any $x\in L$. The first important result on linear (or additive ) commuting maps is Posner¡¯s theorem \cite{Pos} from 1957. Then many scholars study commuting maps on all kinds of algebra structures, Bre$\breve{s}$ar \cite{Bre3} briefly discuss various extensions of the notion of a commuting map. Recent articles about commuting maps can reference \cite{Bou,Bre2,Bre3,CWS,Chen2016,Fran,Hanw,WD1,WD2,XYi}.

Obviously, if $\phi$ on
$L$ is such a map, then $[\phi(x), y] = [x, \phi(y)]$ for any $x, y\in L$. Define $f(x,y)=[\phi(x),y]=[x,\phi(y)]$, then it is easy to check $f$ is a biderivation of $L$.

Using Theorem \ref{maintheo}, we get the following result.

\begin{theorem}
Any linear map $\phi$ on $W(2,2)$ is commuting if and only if there are complex numbers $\lambda, \mu$ and a linear function $\sigma: W(2,2)\rightarrow\mathbb{C}$ such that
$$
\phi(x)=\lambda x +\mu \omega(x)+\sigma(x) c, \ \ \forall x\in W(2,2),
$$
where $\omega$ is a linear map from $W(2,2)$ into it self defined by $\omega(L_m)=H_m, \omega (H_m)=\omega (c)=0$ for all $m\in \mathbb{Z}$.
\end{theorem}

\begin{proof}
Note that $[\omega(x),x]=0$ for all $x\in W(2,2)$, the ``if part'' is easy to verify. We now prove the ``only if'' part.

By the above discuss we see that $f(x,y)=[\phi(x),y]$, $x,y\in W(2,2)$ is a biderivation of $W(2,2)$. It is follows by Theorem  that $[\phi(x),y]=[\lambda x +\mu \omega (x),y]$ for some $ \lambda, \mu \in \mathbb{C}$. Further, we have $[\phi(x)-\lambda x -\mu \omega (x),y]=0$ and then $\phi(x)-\lambda x -\mu \omega (x) \in Z(W(2,2))=\mathbb{C}c$. This means that there is a map $\sigma$ from $W(2,2)$ in to $\mathbb{C}$ such that
$$
\phi(x)-\lambda x -\mu \omega (x)=\sigma(x)c.
$$
It is easy to check that $\sigma$ is linear. The proof is completed.
\end{proof}

Similar to the proof of Theorem 3.5 of \cite{tang2016}, we get by Theorem  \ref{theo-virasoro} the following result.

\begin{theorem}
Any linear map $\phi$ on $Vir$ is commuting if and only if there is a complex number $\lambda$ and a linear function $\sigma: Vir \rightarrow\mathbb{C}$ such that
$$
\phi(x)=\sigma(x)c + \lambda x, \ \ \forall x\in Vir.
$$
\end{theorem}

\begin{remark}
If we let $c=0$ in Virasoro algebra or in W-algebra $W(2,2)$, then using the same method we are able to give the forms of biderivation of Witt algebra or centerless W-algebra $W(2,2)$. The linear commuting maps on them are also described.
\end{remark}

\subsection{Post-Lie algebra}

Post-Lie algebras have been introduced by Valette in connection with the homology of partition
posets and the study of Koszul operads \cite{vela}. As \cite{Burde1} point out, post-Lie algebras are natural common generalization of pre-Lie algebras  and LR-algebras in the geometric context of nil-affine actions of Lie groups. Recently, many authors study some post-Lie algebras and post-Lie algebra structures  \cite{Burde2,Burde1,Mun,pan,tang2014}. In particular, the authors \cite{Burde1} study the commutative post-Lie algebra structure on Lie algebra, they by using the Levi decompositions proved that any commutative post-Lie algebra structure on a (finite) perfect Lie algebra is trivial.  Note that the Virasoro algebra $Vir$ and W-algebra $W(2,2)$ are both infinite dimensional perfect Lie algebras, we naturally want to know  that is also trivial for commutative post-Lie algebra structure on them? By using our Theorems \ref{theo-virasoro} and \ref{maintheo}, we answer yes to this question. Let us recall the definition of commutative post-Lie algebra as follows.

\begin{definition}\label{post}
Let $(L, [, ])$ be a complex Lie algebra. A  commutative post-Lie algebra structure on $L$ is a
$\mathbb{C}$-bilinear product $x\cdot y$ on $L$ satisfying the following identities:
\begin{eqnarray}
&& x \cdot y = y\cdot x \label{post5}\\
&& [x, y] \cdot z =  x \cdot (y \cdot z)-y \cdot (x \cdot z) \label{post6}\\
&& x\cdot [y, z] = [x\cdot y, z] + [y, x \cdot z] \label{post7}
\end{eqnarray}
for all $x, y, z \in V$. We also call $(L, [, ], \cdot)$ to be a commutative post-Lie algebra.
\end{definition}

The following lemma shows the connection of commutative post-Lie algebra and biderivation of Lie algebra, which first was pointed out.

\begin{lemma}\label{postbide}
Suppose that $(L, [, ], \cdot)$ is a commutative post-Lie algebra. If we define a bilinear map $f : L\times L \rightarrow L$ by $f(x,y)=x\cdot y$ for all $x,y\in L$, then $f$ is a biderivation of $L$.
\end{lemma}

\begin{proof}
For any $x,y,z\in L$, by (\ref{post5}) and (\ref{post7}), we deduce that
\begin{eqnarray*}
 f([x,y],z)=[x,y]\cdot z&=&z \cdot [x,y]=[z\cdot x,y]+[x, z\cdot y]\\
&=&[x\cdot z,y]+[x, y\cdot z]=[f(x,z),y]+[x,f(y,z)],\\
f(x,[y,z])=x\cdot [y,z]&=&[x\cdot y, z]+[y, x\cdot z]=[f(x,y),z]+[y,f(x,z)],
\end{eqnarray*}
which inosculates with (\ref{2der}) and (\ref{1der}), as desired.
\end{proof}

We now give the main result of this section as follows.

\begin{theorem}\label{posttheo}
Suppose that $L$ is the Virasoro algebra $Vir$ or W-algebra $W(2,2)$. Then any commutative post-Lie algebra structure on $L$ is trivial. Namely, $x\cdot y=0$ for all $x,y\in L$.
\end{theorem}

\begin{proof}
(a) For Virasoro algebra $Vir$. Suppose that $(Vir, [, ], \cdot)$ is a commutative post-Lie algebra. By Lemma \ref{postbide} and Theorem  \ref{theo-virasoro}, we know that there is a complex number $\lambda$ such that $x\cdot y=\lambda [x,y]$ for all $x,y\in Vir$. Note that the product $\cdot$ is commutative, we have $\lambda [x,y]=\lambda [y,x]$, which deduces $\lambda=0$. In other words, it is trivial.

(b) For W-algebra $W(2,2)$. Suppose that $(W(2,2), [, ], \cdot)$ is a commutative post-Lie algebra. By Lemma \ref{postbide} and Theorem  \ref{maintheo}, we know that there are complex number $\lambda,\mu$ such that $x\cdot y=\lambda [x,y]+\mu[\omega(x),y]$ for all $x,y\in W(2,2)$, where $\omega$ is defined by Theorem  \ref{maintheo}. Since the product $\cdot$ is commutative, we have $\lambda [x,y]+\mu[\omega(x),y]=\lambda [y,x]+\mu[\omega(y),x]$. Note that (\ref{omega}) tells us that $[\omega(x), y]=[x, \omega (y)]$. It follows that $2(\lambda[x,y]+\mu[\omega(x),y])=0$ and so that $\lambda=\mu=0$. That is, $x\cdot y=0$ for all $x,y\in W(2,2)$.
\end{proof}

\begin{remark}
The proof of Theorem \ref{posttheo} means that we are able to characterize the commutative post-Lie algebra structure on a Lie algebra $L$ while we know the forms of biderivation of $L$.  But the precondition is that the biderivation should be assumed to be non-skewsymmetric. Thus, using a result of \cite{tang2016} on biderivations of general linear lie algebras $gl_n(\mathbb{C})$, we find there is a non-trivial commutative post-Lie algebra structure, i.e., the product is given by $x\cdot y= \mu {\rm tr}(x){\rm tr} (y) I_n$ for all $x,y\in gl_n(\mathbb{C})$, where  $I_n$ and ${\rm tr}(x)$ denote the identity matrix and trace of $x$ respectively.
\end{remark}

\section{ACKNOWLEDGMENTS}
This work is supported in part by National Natural Science Foundation of China(Grant No. 11171294), Natural Science
Foundation of Heilongjiang Province of China (Grant No. A2015007), the fund of Heilongjiang Education Committee (Grant No. 12531483).

\end{document}